\documentclass[a4paper,12pt]{article}
\usepackage[T1]{fontenc}
\usepackage[frenchb]{babel}
\usepackage[colorlinks=true, linkcolor=black, urlcolor=blue, filecolor=black, citecolor=black]{hyperref}
\usepackage{graphicx}
\usepackage{amsmath}
\usepackage{amssymb}
\usepackage{amsfonts}
\usepackage{amsfonts}
\usepackage{amsthm}
\usepackage{mathrsfs}
\usepackage{fancyhdr}
\usepackage[latin1]{inputenc}
\usepackage{lipsum}
\usepackage{caption}
\usepackage{appendix}
\usepackage{makeidx}
\usepackage{xcolor}
\usepackage[all]{xy}

\usepackage{geometry}
\newtheorem{thm}{Theorem}[section]

\newtheorem{prop}[thm]{Proposition}
\newtheorem{cor}[thm]{Corollary}

\newtheorem{defi}[thm]{Definition}

\newtheorem{lm}[thm]{Lemma}
\newtheorem{conj}[thm]{Conjecture}

\newtheorem*{thmnum*}{Théorème}
\newtheorem*{propnum*}{Proposition}
\newtheorem*{cornum*}{Corollaire}
\newtheorem*{definum*}{Définition}
\newtheorem*{lmnum*}{Lemme}
\newtheorem*{conjnum*}{Conjecture}
\newtheorem*{exnum*}{Exemple}

\newtheoremstyle{rmk}{\topsep}{\topsep}{\upshape}{0pt}{\bfseries}{ : }{ }{}

\theoremstyle{rmk}

\addto\captionsfrench{}

\newcommand{\N}{\mathbb{N}}
\newcommand{\C}{\mathbb{C}}
\newcommand{\p}{\mathbb{P}}

\newcommand{\cO}{\mathcal{O}}

\author{MABED Yanis}
\title{\textbf{Totally Invariant Divisors of non Trivial Endomorphisms of the Projective Space}}

\date{}
\begin{document}

\maketitle

\vspace{1\baselineskip}

\begin{abstract}
It is expected that a totally invariant divisor of a non-isomorphic endomorphism of the complex projective space is a union of hyperplanes. In this paper, we compute an upper bound for the degree of such a divisor. As a consequence, we prove the linearity of totally invariant divisors with isolated singularities.
\end{abstract}

\vspace{1\baselineskip}

\section{Introduction}

\vspace{1\baselineskip}

Only a few complex projective varieties admit non-isomorphic endomorphisms (which we call non trivial endomorphisms). The projective space of any dimension is one of them and a subset $D\subset \p^{n}$ is said totally invariant if there exists a non trivial endomorphism $f$ of $\p^{n}$ such that $f^{-1}(D)=D$. The following statement has been conjectured in \cite{bcs}

\vspace{1\baselineskip}

\begin{conj}\label{conjbcs}
All totally invariant prime divisors of $\p^{n}$ are linear.
\end{conj}

\vspace{1\baselineskip}

For $n\geq 3$, the conjecture is true for smooth divisors thanks to a result of Beauville \cite[Thm]{beauvilleart} for smooth divisors of degree higher than two and Cerveau-Lins Neto \cite[Thm 2]{cln} for smooth quadrics. Let $D\subset \p^{n}$ be a totally invariant divisor by an non trivial endomorphism $f$ of $\p^{n}$, the logarithmic ramification formula (see \cite[Lemma 2.5]{bh}) gives the following
\begin{equation}
K_{\p^{n}}+D=f^{*}(K_{\p^{n}}+D)+R\nonumber
\end{equation}
where $R$ is an uniquely determined effective divisor. We deduce that the degree $d$ of a totally invariant divisor is bounded by $n+1$. Hwang and Nakayama showed that irreducible hypersurfaces of $\p^{n}$ of degree $n+1$ are not totally invariant \cite[Thm 2.1.]{hn} and Höring proved that no irreducible hypersurface of $\p^{n}$ of degree $n$ is totally invariant (see \cite{horing}). This last statement finishes the proof of the conjecture in $\p^{3}$ thanks to Zhang who showed that no quadric of $\p^{3}$ is totally invariant (see \cite[Thm 1.1.]{dqzhang}). Our goal here is to prove the two following theorems.

\vspace{1\baselineskip}

\begin{thm}\label{thmA}
Let $X$ be an irreducible totally invariant divisor of $\p^{n}$ of degree $d$ and let $l$ be the dimension of the singular locus of $X$ (we set $l=-1$ if $X$ is smooth). Then we have
\begin{equation}
d<\binom{n}{l+1}^{1/(n-l-1)}+1\nonumber.
\end{equation}
\end{thm}

\vspace{1\baselineskip}

The proof of this statement follows the method in \cite{horing} which involves the Chern classes of the logarithmic cotangent sheaf $\Omega^{1}_{\p^{n}}(log\ X)$. To understand the strengths and weaknesses of this bound, we will see multiple corollaries in the conclusion of the paper. 

\vspace{1\baselineskip}

\begin{thm}\label{thmB}
If $X$ is a divisor of $\p^{k}$ which is not totally invariant then, for all $n\geq k$, no cone of $\p^{n}$ over $X$ (see definition $\ref{defcone}$) is totally invariant .
\end{thm}

\vspace{1\baselineskip}

Since singular quadrics can be seen as cones over smooth quadrics, this result implies that any quadric is totally invariant by a non trivial endomorphism. Indeed, smooth quadrics can not be totally invariant by \cite[Thm 2]{cln}. The two theorems above join forces to give the following corollary.

\vspace{1\baselineskip}

\begin{cor}\label{corisolated}
Let $X$ be an irreducible divisor of $\p^{n}$ with isolated singularities. If $X$ is totally invariant then it is a hyperplane.
\end{cor}

\vspace{1\baselineskip}

\textbf{Aknowledgements.} This work was completed during my PhD thesis directed by Andreas Höring and Ekaterina Amerik. I wish also thank Daniele Faenzi and Amaël Broustet for their crucial advices. This thesis was supported by Université Côte d'Azur and by Laboratoire de Mathématiques J.A. Dieudonné.

\vspace{1\baselineskip}

\section{Bound of the degree}

\vspace{1\baselineskip}

Let $X$ be an irreducible divisor of $\p^{n}$ totally invariant by a non trivial endomorphism $f$ of $\p^{n}$ with $n\geq 2$. The original idea of Höring in \cite{horing} is about comparing the following Chern classes
\begin{equation}
c_{2}[\Omega_{\p^{n}}(log\ X)\otimes \mathcal{O}(r)]\nonumber
\end{equation}
for $r\in \N^{*}$. We want to adapt this point of vue to the higher Chern classes of these logarithmic cotangent sheafs. When the variety $X$ is singular, we cannot ensure that these sheafs are locally free everywhere but it's locally free and globally generated wherever $X$ is smooth. That's why the Theorem $\ref{thmA}$ involves the dimension $l$ of the singular locus. We set $k=n-l-1$, so we have
\begin{equation}
codim_{\p^{n}} X_{sing}=k+1\nonumber
\end{equation}
and we want to show the following inequality
\begin{equation}
d<\binom{n}{k}^{1/k}+1.\nonumber
\end{equation}
If $X$ is not normal (i.e. $k=1$), one remarks that the bound becomes $n+1$ which is already known. So we can suppose, without any loss of generality, that $X$ is a normal divisor of $\p^{n}$.

\vspace{1\baselineskip}

We first recall the lemma $2.2$ of \cite{horing} :

\vspace{1\baselineskip}

\begin{prop}\label{prop1}
Denoting by $h$ the equation of $X$ in $\p^{n}$ and assuming that $deg\ X \geq 2$, the sheaf $\Omega_{\mathbb{P}^{n}}(log\ X)\otimes \mathcal{O}_{\mathbb{P}^{n}}(1)$ is generated by the following global sections
\begin{equation}\label{eq1}
\frac{d(X_{0}h)}{h},...\ ,\frac{d(X_{n}h)}{h}
\end{equation}
on the complement of $X_{sing}$.
\end{prop}

\vspace{1\baselineskip}

The key point of the proof of Theorem $\ref{thmA}$ is the following lemma which permits to compare Chern classes of sheafs.

\vspace{1\baselineskip}

\begin{lm}\label{keylm}
Let $M$ be a smooth projective variety of dimension $k$, $\mathcal{E}_{1}$ and $\mathcal{E}_{2}$ be two vector bundles of rank $n\geq k$ on $M$, $V\subset H^{0}(M,\mathcal{E}_{1})$ be a vector subspace of dimension $>n$ and
\begin{equation}
\varphi : \mathcal{E}_{1}\rightarrow \mathcal{E}_{2} \nonumber
\end{equation}
an injective sheaf morphism on $M$. We assume the following :
\begin{description}
\item[(i)] the evaluation morphism $ev:V\otimes \mathcal{O}_{M}\rightarrow \mathcal{E}_{1}$ is surjective,
\item[(ii)] for all $x\in M$, one has $rg\ \varphi_{x}\geq n-k$, and
\item[(iii)] for all $1\leq i\leq k+1$, one has $dim\ \{x\in M|rg\ \varphi_{x}\leq n-i+1\}\leq k-i+1$.
\end{description}
Then we have
\begin{equation}
c_{k}(\mathcal{E}_{1})\leq c_{k}(\mathcal{E}_{2}). \nonumber
\end{equation}
\end{lm}

\vspace{1\baselineskip}

\begin{proof}[Proof]
One denotes by $|V|$ the projective space obtained from the vector space $V$ and we consider the closed set
\begin{equation}
B:=\{(x,\sigma)\in M\times |V|\ | (\varphi\circ ev)_{x}(\sigma(x))=0\}\nonumber
\end{equation}
with its two natural projections
\begin{equation}
p_{1}:B\rightarrow M,\ p_{2}: B\rightarrow |V|. \nonumber
\end{equation}
One sets, for all $1\leq i\leq k+1$,
\begin{equation}
R_{i}:= \{x\in M| rg\ \varphi_{x}=n-i+1\}\nonumber
\end{equation}
with $\cup_{i=1}^{k+1}R_{i}=M$ by $(ii)$.\\
If $x\in R_{i}$ then $rg(\varphi\circ ev)_{x}=n-i+1$ and so
\begin{equation}
dim\ p_{1}^{-1}(x)=dim\ |V|-n+i-1.\nonumber
\end{equation}
However, by $(iii)$, we have $dim\ R_{i}\leq k-i+1$ and the dimension of the irreducible components of $B$ are bounded by $dim\ |V|-n+k$.
We finish the proof by induction on $n\geq k$:\\
\indent \underline{\emph{If $n=k$,}} then the dimension of the irreducible components of $B$ are bounded by $dim\ |V|$ so the generic fibers of $p_{2}$ are finite. We deduce that for a generic section $\sigma$ of $\mathcal{E}_{1}$, the section $\varphi\circ \sigma$ of $\mathcal{E}_{2}$ has only isolated zeros. Since $\varphi$ is injective, the generic section $\sigma$ has also isolated zeros and we naturally obtain $c_{k}(\mathcal{E}_{1})\leq c_{k}(\mathcal{E}_{2})$.\\
\indent \underline{\emph{If $n>k$,}} then the dimension of the irreducible components of $B$ are bounded by $dim\ |V|-n+k<dim\ |V|$ so the generic fibers of $p_{2}$ are empty. Therefore, for a generic section $\sigma$ of $\mathcal{E}_{1}$, the section $\varphi\circ\sigma$ never vanishes and neither does $\sigma$. These two sections $\sigma$ and $\varphi\circ\sigma$ define a trivial subbundle of $\mathcal{E}_{1}$ and $\mathcal{E}_{2}$ respectively. Furthermore, the quotients $\mathcal{E}_{1}/\mathcal{O}_{M}$ and $\mathcal{E}_{2}/\mathcal{O}_{M}$, the global section subspace $V/\mathbb{C}\sigma$ and the induced morphism
\begin{equation}
\widetilde{\varphi}: \mathcal{E}_{1}/\mathcal{O}_{M}\rightarrow \mathcal{E}_{2}/\mathcal{O}_{M}\nonumber
\end{equation}
satisfy the lemma's hypothesis so we can apply the induction hypothesis. Finally, we know that $c(\mathcal{E}_{i})=c(\mathcal{E}_{i}/\mathcal{O}_{M})$ because of the following short exact sequence
\begin{equation}
0\rightarrow \mathcal{O}_{M}\rightarrow \mathcal{E}_{i}\rightarrow \mathcal{E}_{i}/\mathcal{O}_{M}\rightarrow 0.\nonumber
\end{equation}
In particular we obtain $c_{k}(\mathcal{E}_{i})=c_{k}(\mathcal{E}_{i}/\mathcal{O}_{M})$ that finishes the proof.
\end{proof}

\vspace{1\baselineskip}

Before applying the previous lemma, we shall recall the Proposition III.10.6. in \cite{hartshorne}. 

\vspace{1\baselineskip}

\begin{prop}\label{prophartshorne}
If $f:X\rightarrow Y$ is a finite morphism between two $n$-dimensional varieties then
\begin{equation}
codim_{X}\{x\in X| rg\ d_{x}f\leq n-i\}\geq i.\nonumber
\end{equation}
\end{prop}

\vspace{1\baselineskip}

We now set $P\subset \mathbb{P}^{n}$ a general $k$-dimensional plane such that $M=f^{-1}(P)$ is a smooth $k$-dimensional variety and $M\cap X_{sing}=\varnothing$. We also set
\begin{equation}
\varphi : f^{*}(\Omega_{\mathbb{P}^{n}}(log\ X)\otimes \mathcal{O}_{\mathbb{P}^{n}}(1))\otimes \mathcal{O}_{M}\rightarrow \Omega_{\mathbb{P}^{n}}(log\ X)\otimes \mathcal{O}_{\mathbb{P}^{n}}(m)\otimes \mathcal{O}_{M}\nonumber
\end{equation}
the injective sheaf morphism on $M$ induced by the logarithmic differential of $f$. The integer $m$ is the algebraic degree of $f$, we have $m\geq 2$ and $f^{*}(\mathcal{O}(1))=\mathcal{O}(m)$. Finally we denote by $V\subset H^{0}(M, \Omega_{\mathbb{P}^{n}}(log\ X)\otimes \mathcal{O}_{\mathbb{P}^{n}}(1)\otimes \mathcal{O}_{M})$ the subspace generated by the restrictions to $M$ of the global sections defined by $(\ref{eq1})$. Ones remarks that $M$ is general $k$-dimensional variety in the following sense : if $V$ is a variety with $dim\ V\leq n-k-1$ then we can assume that $M\cap V=\varnothing$.

\vspace{1\baselineskip}

\begin{prop}\label{propkeylm}
All the objects defined above satisfy the hypothesis of Lemma \ref{keylm}.
\end{prop}

\begin{proof}[Proof]
\textcolor{white}{a}
\begin{description}
\item[(i)] This follows directly from the Proposition $\ref{prop1}$.
\item[(ii)] If $x\in M\backslash X$ then $\Omega_{\mathbb{P}^{n}}(log\ X)_{x}\simeq \Omega_{\mathbb{P}^{n}}$ and the logarithmic differential is only the classic differential. Since $f$ is finite, we deduce from Proposition $\ref{prophartshorne}$ : 
\begin{equation}
dim\ \{rg\ (df_{log})_{x}\leq n-k-1\} \leq n-k-1\nonumber
\end{equation}
then, since the $k$-variety $M$ was chosen in a generic way, we have
\begin{equation}
M\cap\{rg\ (df_{log})_{x}\leq n-k-1\}=\varnothing.\nonumber
\end{equation}
so for all $x\in M\backslash X$, $rg\ \varphi_{x}\geq n-k$.\\
If $x\in X\backslash X_{sing}$ then we have $\Omega_{X,x}\subset \Omega_{\mathbb{P}^{n}}(log\ X)_{x}$. Indeed, if $\{u_{1}=0\}$ is a local equation of $X$ near $x$ so the $\C$-vector space $\Omega_{\mathbb{P}^{n}}(log\ X)_{x}$ admits
\begin{equation}
\frac{du_{1}}{u_{1}},du_{2},...,du_{n}\nonumber
\end{equation}
as a basis and $\Omega_{X,x}$ has
\begin{equation}
du_{2},...,du_{n}\nonumber
\end{equation}
as a basis. Furthermore, since $X$ is totally invariant by $f$, we have the following commutative diagram where the vertical arrows are injective :

\begin{center} $
\xymatrix{
(f|_{X})^{*}\Omega_{X,x}\ar[dd] \ar[rr]^{d(f|_{X})} & & \Omega_{X,x}\ar[dd] \\
\\
f^{*}\Omega_{\mathbb{P}^{n}}(log\ X)_{x} \ar[rr]^{df_{log}} & & \Omega_{\mathbb{P}^{n}}(log\ X)_{x}
}$
\end{center}
But, since $f|_{X}$ is finite, we have
\begin{equation}
dim\ \{rg\ d(f|_{X})\leq n-k-1\}\leq n-k-1\nonumber
\end{equation}
so, again because the $k$-dimensional variety $M$ was chosen in a generic way, we obtain
\begin{equation}
M\cap\{rg\ d(f|_{X})_{x}\leq n-k-1\}=\varnothing.\nonumber
\end{equation}
In other words, the rank of $d(f|_{X})$ on $M\cap X$ is at least $n-k$ and by the previous diagram, so is the rank of $df_{log}$ on $M\cap X$.
\item[(iii)] We fix $1\leq i\leq k+1$. For $x$ outside of $X$ and by the same reason as $(ii)$, we have
\begin{equation}
dim\{rg\ (df_{log})_{x}\leq n-i+1\}\leq n-i+1
\end{equation}
and so
\begin{equation}
dim\ M\cap\{rg\ (df_{log})_{x}\}\leq n-i+1\}\leq k-i+1.\nonumber
\end{equation}
Similarly, on $X$ we have
\begin{equation}
dim\ M\cap\{d(f|_{X})_{x}\leq n-i\}\leq k-i.\nonumber
\end{equation}
Then, by reusing the diagram in $(ii)$, we find that the dimension of $\{rg\ (df_{log})_{x}\leq n-i+1\}$ is at most $k-i+1$ and finally
\begin{equation}
dim\ \{rg\ \varphi_{x}\leq n-i+1\}\leq k-i+1.\nonumber
\end{equation}
\end{description}
\end{proof}

\vspace{1\baselineskip}

As a consequence of the key lemma, we obtain the following.

\vspace{1\baselineskip}

\begin{prop}\label{propinegchern}
Let $X$ be a prime divisor of $\p^{n}$ which is totally invariant by a non trivial endomorphism $f$ of $\p^{n}$. We set $k=\mbox{codim}_{X}X_{sing}$ and $M$ defined as in Proposition $\ref{propkeylm}$. Then we have the following inequality :
\begin{equation}\label{eq3}
c_{k}(f^{*}(\Omega_{\mathbb{P}^{n}}(log\ X)(1))\otimes \mathcal{O}_{M})\leq c_{k}(\Omega_{\mathbb{P}^{n}}(log\ X)(m)\otimes \mathcal{O}_{M}).
\end{equation}
\end{prop}

\vspace{1\baselineskip}

This brings us to the computation of $c_{k}(\Omega_{\mathbb{P}^{n}}(log\ X)\otimes \mathcal{O}_{\mathbb{P}^{n}}(m))$, the result is summed up in the next proposition.

\vspace{1\baselineskip}

\begin{prop}
The $k$-th Chern class of $\Omega_{\mathbb{P}^{n}}(log\ X)\otimes \mathcal{O}_{\mathbb{P}^{n}}(1)$ does not depend of $n$. More precisely, we have
\begin{equation}
c_{k}(\Omega_{\mathbb{P}^{n}}(log\ X)\otimes \mathcal{O}_{\mathbb{P}^{n}}(1))=(d-1)^{k}H^{k}.\nonumber
\end{equation}
Furthermore, we have the following estimate
\begin{equation}
c_{k}(\Omega_{\mathbb{P}^{n}}(log\ X)\otimes \mathcal{O}_{\mathbb{P}^{n}}(m))=[\binom{n}{k}m^{k}+\binom{n-1}{k-1}(d-n-1)m^{k-1}+o(m^{k-1})]H^{k}.\nonumber
\end{equation}
\end{prop}

\vspace{1\baselineskip}

\begin{proof}[Proof]
As a first step, we shall compute $c_{i}(\Omega_{\mathbb{P}^{n}}(log\ X))$ for $i=1,...,k$ by using the short exact sequence $(2.8)$ of \cite{dolgachev}
\begin{equation}
0\rightarrow \Omega_{\p^{n}}^{1}\rightarrow \widetilde{\Omega}^{1}_{\p^{n}}(log\ X)\overset{res}{\rightarrow}\mu_{*}\cO_{\widetilde{X}}\rightarrow 0.
\end{equation}
where $\mu : \widetilde{X}\rightarrow X$ is a desingularization of $X$. 
Since $X$ is a normal variety and by \cite[1.13.]{debarre}, we have $\mu_{*}\cO_{\widetilde{X}}\simeq \cO_{X}$. We deduce that
\begin{equation}
c_{i}(\mu_{*}\mathcal{O}_{\widetilde{X}})=c_{i}(\mathcal{O}_{X})=X^{i}=d^{i}H^{i}\nonumber
\end{equation}
for all $1\leq i\leq k$ then
\begin{equation}
c_{i}(\Omega_{\mathbb{P}^{n}}(log\ X))=(\sum_{j=0}^{i}(-1)^{j}\binom{n+1}{j}d^{i-j})H^{i}\nonumber
\end{equation}
for $i=1,...,k$.\\
We now use the lemma $A.2.1$ from \cite{these} with $\mathcal{E}=\Omega_{\mathbb{P}^{n}}(log\ X)$ and $\mathcal{L}=\mathcal{O}_{\mathbb{P}^{n}}(m)$ to find the following
\begin{eqnarray}
c_{k}(\Omega_{\mathbb{P}^{n}}(log\ X)\otimes \mathcal{O}_{\mathbb{P}^{n}}(m))&=&[\sum_{i=0}^{k}m^{k-i}(\sum_{j=0}^{i}(-1)^{j}\binom{n+1}{j}\binom{n-i}{k-i}d^{i-j})]H^{k}\nonumber\\
&=&[\binom{n}{k}m^{k}+\binom{n-1}{k-1}(d-n-1)m^{k-1}+o(m^{k-1})]H^{k}\nonumber
\end{eqnarray}
and, in particular for $m=1$,
\begin{eqnarray}
c_{k}(\Omega_{\mathbb{P}^{n}}(log\ X)\otimes \mathcal{O}_{\mathbb{P}^{n}}(1))&=&[\sum_{i=0}^{k}\sum_{j=0}^{i}(-1)^{j}\binom{n+1}{j}\binom{n-i}{k-i}d^{i-j}]H^{k}\nonumber\\
&=&[\sum_{j=0}^{k}\sum_{i=j}^{k}(-1)^{j}\binom{n+1}{j}\binom{n-i}{k-i}d^{i-j}]H^{k}\nonumber\\
&=&[\sum_{j=0}^{k}\sum_{i=0}^{k-j}(-1)^{j}\binom{n+1}{j}\binom{n-i-j}{k-i-j}d^{i}]H^{k}\nonumber\\
&=&[\sum_{i=0}^{k}[\sum_{j=0}^{k-i}(-1)^{j}\binom{n+1}{j}\binom{n-i-j}{k-i-j}]d^{i}]H^{k}.\nonumber
\end{eqnarray}
After the index change $i=k-i$ in the first summand, we obtain
\begin{equation}
c_{k}(\Omega_{\mathbb{P}^{n}}(log\ X)\otimes \mathcal{O}_{\mathbb{P}^{n}}(1))=[\sum_{i=0}^{k}[\sum_{j=0}^{i}(-1)^{j}\binom{n+1}{j}\binom{n-k+i-j}{i-j}]d^{k-i}]H^{k}\nonumber
\end{equation}
Now for proving that $c_{k}(\Omega_{\mathbb{P}^{n}}(log\ X)\otimes \mathcal{O}_{\mathbb{P}^{n}}(1))=(d-1)^{k}=\sum_{i=0}^{k}\binom{k}{i}(-1)^{i}d^{k-i}$, it suffices to show that for all $0\leq i\leq k$,
\begin{equation}
\sum_{j=0}^{i}(-1)^{j}\binom{n+1}{j}\binom{n-k+i-j}{i-j}=(-1)^{i}\binom{k}{i}.\nonumber
\end{equation}
We denote this property by $\mathcal{P}(n,k,i)$ for $1\leq i\leq k\leq n$.\\
\indent One first shows that $\mathcal{P}(n,k,k)$ for all $k$ :
\begin{equation}
\sum_{j=0}^{k}(-1)^{j}\binom{n+1}{j}\binom{n-j}{k-j}=(-1)^{k}\nonumber
\end{equation}
Now let's do an induction on $n\geq 1$ : the $n=1$ case is trivial and 
\begin{eqnarray}
\sum_{j=0}^{k}(-1)^{j}\binom{n+1}{j}\binom{n-j}{k-j}&=&\sum_{j=0}^{k}(-1)^{j}\binom{n}{j}\binom{n-j}{k-j}+\sum_{j=1}^{k}(-1)^{j}\binom{n}{j-1}\binom{n-j}{k-j}\nonumber\\
&=&\sum_{j=0}^{k}(-1)^{j}\binom{n}{k}\binom{k}{j}-\sum_{j=0}^{k-1}(-1)^{j}\binom{n}{j}\binom{n-j-1}{k-j-1}\nonumber\\
&=&0-(-1)^{k-1}\ by\ \mathcal{P}(n-1,k-1,k-1)\nonumber\\
&=&(-1)^{k}.\nonumber
\end{eqnarray}
Fixing $n$, we now show $\mathcal{P}(n,k,i)$ by an induction on $k\geq 1$ : the $k=1$ case is trivial so we assume that $\mathcal{P}(n,k,i)$ is true for all $i\leq k$ and it suffices to show that $\mathcal{P}(n,k+1,i)$ is true for all $i\leq k$ since $\mathcal{P}(n,k+1,k+1)$ has been shown previously.

\begin{eqnarray}
\sum_{j=0}^{i}(-1)^{j}\binom{n+1}{j}\binom{n-k-1+i-j}{i-j}&=&\sum_{j=0}^{i}(-1)^{i-j}\binom{n+1}{i-j}\binom{n-k-1+j}{j}\nonumber\\
&=&\sum_{j=0}^{i}(-1)^{i-j}\binom{n+1}{i-j}\binom{n-k+j}{j}\nonumber\\
&&-\sum_{j=1}^{i}(-1)^{i-j}\binom{n+1}{i-j}\binom{n-k-1+j}{j-1}\nonumber\\
&=&(-1)^{i}\binom{k}{i}\nonumber\\
&&-\sum_{j=0}^{i-1}(-1)^{i-j-1}\binom{n+1}{i-j-1}\binom{n-k+j}{j}\ par\ \mathcal{P}(n,k,i)\nonumber\\
&=&(-1)^{i}\binom{k}{i}+(-1)^{i}\binom{k}{i-1}\ par\ \mathcal{P}(n,k,i-1)\nonumber\\
&=&(-1)^{i}\binom{k+1}{i}\nonumber
\end{eqnarray}
and the statement follows.
\end{proof}

\vspace{1\baselineskip}

Now we have computed the Chern classes, we can rewrite $(\ref{eq3})$ in Proposition $\ref{propinegchern}$.

\vspace{1\baselineskip} 

\begin{proof}[Proof of theorem $\ref{thmA}$]
First of all, we have
\begin{eqnarray}
c_{k}(f^{*}(\Omega_{\mathbb{P}^{n}}(log\ X)\otimes \mathcal{O}_{\mathbb{P}^{n}}(1))\otimes \mathcal{O}_{M})&=&f^{*}((d-1)^{k}H^{k}).M\nonumber\\
&=&(d-1)^{k}m^{k}H^{k}.M\nonumber\\
&=&(d-1)^{k}m^{n}\in \mathbb{Z}\nonumber
\end{eqnarray}
because $M=f^{*}(H^{n-k})=m^{n-k}H^{n-k}$ in the cohomology ring of $\mathbb{P}^{n}$. And on the other hand,
\begin{equation}
c_{k}(\Omega_{\mathbb{P}^{n}}(log\ X)\otimes \mathcal{O}_{\mathbb{P}^{n}}(m)\otimes \mathcal{O}_{M})=\binom{n}{k}m^{n}+\binom{n-1}{k-1}(d-n-1)m^{n-1}+o(m^{n-1}).\nonumber
\end{equation}
The inequality $(\ref{eq3})$ becomes 
\begin{equation}\label{eq4}
\binom{n}{k}m^{n}+\binom{n-1}{k-1}(d-n-1)m^{n-1}+o(m^{n-1}) \geq (d-1)^{k}m^{n}.\nonumber
\end{equation}
However, when we take iterates of $f$, the integer $m$ goes to infinity and we have
\begin{equation}
(d-1)^{k}\leq \binom{n}{k}.\nonumber
\end{equation}
Furthermore, if $(d-1)^{k}=\binom{n}{k}$, then $d=n+1$ since $\binom{n-1}{k-1}(d-n-1)\leq 0$. This contradicts $(d-1)^{k}=\binom{n}{k}$ for $k\geq 2$. Finally we have shown
\begin{equation}
(d-1)^{k}<\binom{n}{k}\nonumber.
\end{equation}
\end{proof}

\vspace{1\baselineskip}

\section{The cone case}

\vspace{1\baselineskip}

This section is devoted to the study of totally invariant cones of $\p^{n}$ over a hypersurface $X$ of $\p^{k}$ with $k\leq n$. More precisely, for such a cone, we show that the hypersurface $X$ is either totally invariant in $\p^{k}$. First of all, we shall give the precise definition of a cone.

\vspace{1\baselineskip}

\begin{defi}\label{defcone}
Let $X$ be a divisor of $\p^{k}$ whose equation is $h\in \C[X_{0},...,X_{k}]$ and let $n\geq k$ be an integer, we define the cone of $\p^{n}$ over $X$ as the divisor of $\p^{n}$ whose equation is $h\in \C[X_{0},...,X_{n}]$.
\end{defi}

\vspace{1\baselineskip}

\begin{proof}[Proof of $\ref{thmB}$]
\indent Let $X$ be a divisor of $\p^{k}$ which is not totally invariant and whose equation is $h\in\C[X_{0},...,X_{k}]$. We have to show that no cone $C$ over $X$ is totally invariant. The proof proceeds by induction on the integer $n\geq k$. If $n=k$ then the cone $C$ is equal to $X$ and there is nothing to show. If $n>k$ one assumes by contradiction that there exists a non trivial endomorphism $f=(F_{0}:...:F_{n})$ of $\mathbb{P}^{n}$ with algebraic degree $q$ such that
\begin{equation}
f^{-1}(C)=C.\nonumber
\end{equation}
The total invariance of $C$ by $f$ is equivalent to the following (see \cite[Lemma 4.3.]{fs}) : there exists $\lambda\in \mathbb{C}^{*}$ such that
\begin{equation}\label{eq93}
h(F_{0},...,F_{k})=\lambda h(X_{0},...,X_{k})^{q}.
\end{equation}
Our goal is to find a non trivial endomorphism $g$ of $\p^{n-1}$ whose the cone $C'$ of $\p^{n-1}$ over $X$ is totally invariant. A first idea would be considering the rational map $(F_{0}|_{\p^{n-1}}:...:F_{n-1}|_{\p^{n-1}})$ but we meet an issue : this map is not necessarily a morphism because the $F_{i}|_{\p^{n-1}}$ for $i=1,...,n-1$ can have common zeros. However, we can solve this problem with the following method :\\

We set $P=(0:...:0:1)\in \p^{n}$. Since $f$ is a finite morphism, there exists an hyperplane $H\subset \mathbb{P}^{n}$ such that
\begin{equation}
H\cap (f^{-1}(P)\cup\{P\})=\varnothing\nonumber.
\end{equation}
Since $H=\{a_{0}X_{0}+...+a_{n}X_{n}=0\}$ does not contain the point $P$, we have $a_{n}\neq 0$ and we can assume that
\begin{equation}
H=\{X_{n}=\sum_{i=0}^{n}a_{i}X_{i}\}\nonumber.
\end{equation}
We now consider the following endomorphism of $H$ :
\begin{equation}
g:=\pi_{P,H}\circ f|_{H}\nonumber
\end{equation}
where $\pi_{P,H}$ is the projection from $P$ over $H$ whose equations are 
\begin{equation}
\pi_{P,H}:(x_{0}:...:x_{n})\mapsto(x_{0}:...:x_{n-1}:\sum_{i=0}^{n-1}a_{i}x_{i}).\nonumber
\end{equation}
It is well defined on $\mathbb{P}^{n}\backslash \{P\}$ and so $g$ is not well defined on $x\in H$ if and only if $f(x)=P$. This never happens because $H\cap f^{-1}(P)=\varnothing$ and then the map $g$ is an endomorphism of $H$.\\
Furthermore, $H\cap C$ is totally invariant by $g$ : indeed, when we restrict $(\ref{eq93})$ on $H$, we find
\begin{equation}
h(F_{0}|_{H},...,F_{k}|_{H})=\lambda h(X_{0},...,X_{k})^{q}
\end{equation}
where $h$ is the equation of $C\cap H$ and $F_{0}|_{H},...,F_{k}|_{H}$ are the $k$ first equations of the endomorphism $g$. The inclusion $C\cap H\subset H$ is isomorphic to the inclusion $C'\subset \mathbb{P}^{n-1}$ where $C'$ is the cone of $\p^{n-1}$ over $X\subset \mathbb{P}^{k}$. This is a contradiction to the induction hypothesis and then $C\subset \mathbb{P}^{n}$ is not totally invariant.
\end{proof}

\vspace{1\baselineskip}

Thanks to bilinear algebra, it is well-known that all singular irreducible quadrics are cones over a smooth irreducible quadric. These latter aren't totally invariant thanks to Cerveau-Lins Neto \cite[Thm 2]{cln} so we directly obtain the following.

\vspace{1\baselineskip}

\begin{cor}\label{corquadric}
Let $X$ be a prime divisor of $\p^{n}$ which is totally invariant. Then $X$ is not a quadric.
\end{cor}

\vspace{1\baselineskip}

\section{Conclusion}

\vspace{1\baselineskip}

To conclude this paper, we give some precisions about the impact of Theorems $\ref{thmA}$ and $\ref{thmB}$ on the linearity conjecture. The main effect of this bound appear on totally invariant divisors with isolated singularities as we already present in the corollary $\ref{corisolated}$. Let $X$ be a totally invariant prime divisor of $\p^{n}$ and fix the dimension $l$ of the singular locus of $X$. In a natural way, we can wonder for which $n$ the conjecture is verified. The following statement is devoted to this question.

\vspace{1\baselineskip}

\begin{cor}
Let $X$ be a totally invariant prime divisor of $\p^{n}$ with $n\geq 4$. In the Table $\ref{tabverifconj}$ we have registered for which $n$ the conjecture is true depending on the dimension $l$ of the singular locus of $X$.
\begin{table}[h]
\hspace*{-1cm}\begin{tabular}{|c|c|c|c|c|c|c|c|c|}
\hline
Values of $l$ & $-1$ & $0$ & $1$ & $2$ & $3$ & $4$ & $5$ & $6$\\
\hline
The conjecture is true if & $n\geq 4$ & $n\geq 4$ & $n\geq 6$ & $n\geq 10$ & $n\geq 14$ & $n\geq 19$ & $n\geq 23$ & $n\geq 27$\\
\hline
\end{tabular}
\caption{}
\label{tabverifconj}
\end{table}
\end{cor}

\vspace{1\baselineskip}

\begin{proof}[Proof]
Assume that the divisor $X$ is not a hyperplane so Theorem $\ref{thmB}$ gives that the degree $d$ of $X$ is necessarily at least $3$. According to Theorem $\ref{thmA}$, we find
\begin{equation}
2^{n-l-1}\leq (d-1)^{n-l-1}<\binom{l+1}{n}\nonumber.
\end{equation}
We deduce that, fixing the integer $l$, the conjecture is true when the function
\begin{equation}
\phi_{l}(n)=2^{n-l-1}-\binom{l+1}{n}\nonumber
\end{equation}
is non negative. For each value of $l$ in the table above, it suffices to find the values of $n$ such that $\phi_{l}(n)$ is non negative. A simple study of the function $\phi_{l}$ in each case gives the results of the table.
\end{proof}

\vspace{1\baselineskip}

We deduce from the previous corollary that in order to prove conjecture for $\p^{4}$ and $\p^{5}$, it is sufficient to verify it for divisors with $1$-dimensional singular locus.

\vspace{1\baselineskip}

We now wonder what happens for totally invariant prime divisors of $\p^{n}$ with big degrees (recall that the degree of such a divisor is bounded by $n-1$). Let's begin with the general case.

\vspace{1\baselineskip}

\begin{cor}\label{cornonnorm}
Let $X$ be a totally invariant prime divisor of $\p^{n}$ and let $l$ be the dimension of the singular locus of $X$. Assume that the degree of $X$ is $n-k$ for an integer $1\leq k\leq n-3$ then the divisor $X$ is necessarily non normal whenever $n\geq n_{k}:=2k+3/2+\sqrt{2k^{2}+2k+1/4}$.
\end{cor}

\vspace{1\baselineskip}

\begin{proof}[Proof]
Let assume that $l=n-3$ then the theorem $\ref{thmA}$ gives the following inequality
\begin{equation}
(n-k-1)^{2}-\frac{n(n-1)}{2}<0\nonumber.
\end{equation}
This is a quadratic polynomial whose greater root is exactly
\begin{equation}
n_{k}=2k+3/2+\sqrt{2k^{2}+2k+1/4}.\nonumber
\end{equation}
We conclude that the latter inequality is no longer true when $n\geq n_{k}$.
\end{proof}

\vspace{1\baselineskip}

In order to give a meaning to the above corollary, we shall precise the two special cases below : whom involve respectively the $n-1$ and $n-2$ degree case.

\vspace{1\baselineskip}

\begin{cor}
Let $X$ be a totally invariant prime divisor of $\p^{n}$ with $n\geq 4$ and let $l$ be the dimension of the singular locus of $X$. Assume that the degree of $X$ is $n-1$ then :
\begin{description}
\item[(1)] the dimension $l$ is $n-2$ or $n-3$ for $n=4,5$ ; and
\item[(2)] the divisor $X$ is necessarily non normal when $n\geq 6$.
\end{description}
\end{cor}

\vspace{1\baselineskip}

\begin{proof}[Proof]
Theorem $\ref{thmA}$ gives the following inequality
\begin{equation}
(n-2)^{n-l-1}<\binom{l+1}{n}\nonumber.
\end{equation}
The first assertion is proven by a direct computation of the latter inequality in the special cases $n=4$ and $n=5$. For the second statement, one can apply Corollary $\ref{cornonnorm}$ with $k=1$. Indeed, one has$n_{1}\simeq 5,56$.
\end{proof}

\vspace{1\baselineskip}

We can ask the same question to divisors of degree $n-2$. The following corollary answer to that question.

\vspace{1\baselineskip}

\begin{cor}
Let $X$ be a totally invariant prime divisor of $\p^{n}$ with $n\geq 5$ and let $l$ be the dimension of the singular locus of $X$. Assume that the degree of $X$ is $n-2$ then :
\begin{description}
\item[(1)] the dimension $l$ is $n-2$ or $n-3$ for $n=6,7,8$ ; and
\item[(2)] the divisor $X$ is necessarily non normal when $n\geq 9$.
\end{description}
\end{cor}

\vspace{1\baselineskip}

\begin{proof}[Proof]
Theorem $\ref{thmA}$ gives the following inequality
\begin{equation}
(n-3)^{n-l-1}<\binom{l+1}{n}\nonumber.
\end{equation}
The first assertion is proven by a direct computation of the latter inequality in the special cases $n=6,7,8$. For the second statement, one can apply Corollary $\ref{cornonnorm}$ with $k=2$. Indeed, one has $n_{2}=9$.
\end{proof}

\bibliographystyle{alpha}
\bibliography{biblio}
\addcontentsline{toc}{section}{Bibliographie}

\vspace{1\baselineskip}

\textsc{Yanis Mabed, Université Côte d'Azur, CNRS, LJAD, France}\\
\textit{Email address :} \textbf{yanis.mabed@unice.fr ; ymabed@gmail.com}

\end{document}